\documentclass[pdflatex,sn-mathphys-num]{sn-jnl}

\usepackage{graphicx}
\usepackage{multirow}
\usepackage{amsmath,amssymb,amsfonts}
\usepackage{amsthm}
\usepackage{mathrsfs}
\usepackage[title]{appendix}
\usepackage{xcolor}
\usepackage{textcomp}
\usepackage{manyfoot}
\usepackage{booktabs}
\usepackage{algorithm}
\usepackage{algorithmicx}
\usepackage{algpseudocode}
\usepackage{mathtools}
\usepackage{tikz}
\usepackage{pgfplots}
\usetikzlibrary{arrows.meta,positioning,shapes.geometric}
\usepgfplotslibrary{groupplots}
\pgfplotsset{compat=1.18}
\mathtoolsset{showonlyrefs}

\theoremstyle{thmstyleone}
\newtheorem{theorem}{Theorem}
\newtheorem{proposition}[theorem]{Proposition}

\theoremstyle{thmstyletwo}
\newtheorem{remark}[theorem]{Remark}

\begin{document}

\title[Iterative Refinement for Non-Hermitian Eigendecompositions]{Iterative Refinement for Diagonalizable Non-Hermitian Eigendecompositions}

\author*[1]{\fnm{Takeshi} \sur{Terao}}
\email{takeshi.terao@aoni.waseda.jp}
\affil*[1]{\orgdiv{Research Institute for Science and Engineering}, \orgname{Waseda University}, \orgaddress{\street{3-4-1, Okubo}, \city{Shinjuku-ku}, \postcode{169-8555}, \state{Tokyo}, \country{Japan}}}

\abstract{This paper develops matrix-multiplication-based iterative refinement for diagonalizable non-Hermitian eigendecompositions. The main theory concerns simple eigenvalues and distinguishes two input regimes. In the right-only regime, where only approximate right eigenvectors and eigenvalues are available, a first-order derivation selects the update and the resulting post-update residual identity is exact, yielding a quadratic residual bound. In the left-right regime, where approximate left and right eigenvectors are both available, the computable driving matrix is an exact perturbation of the inverse-based one and the biorthogonality correction satisfies an exact Newton--Schulz-type error identity. Under a small biorthogonality error, these relations yield a local second-order estimate for the resulting $W$-method. Clustered eigenvalues are handled separately by a stabilization extension based on clusterwise re-diagonalization and suppression of intracluster corrections, whose effect is verified on controlled matrices with ill-conditioned cluster bases. The method is intended as post-processing for an already accurate eigendecomposition. The attraction region is not analyzed, and no complete theory is given for the clustered case.}


\maketitle

\section{Introduction}\label{sec:intro}

For non-Hermitian eigenvalue problems, finite-precision errors are often amplified by nonnormality and by the lack of orthogonality of the eigenvectors \cite{wilkinson1965algebraic,stewart1990matrix,kato1995perturbation,trefethen2005spectra}. A basic question is therefore how far a previously computed eigendecomposition can be improved by iterative refinement. The present paper addresses this question in a post-processing setting: an approximate eigendecomposition is already available, and the goal is to improve it at $\mathcal{O}(n^3)$ cost through matrix-multiplication-dominant updates, without recomputing the decomposition from scratch. Throughout the theoretical sections we work in exact arithmetic; finite-precision effects are discussed only in the numerical experiments. In contrast to the real symmetric or Hermitian case, the answer depends strongly on which auxiliary information accompanies the approximate eigenpairs.

Classical references for numerical methods for nonsymmetric eigenproblems include \cite{bai2000templates,saad2011numerical,sorensen2002numerical,golub2013matrix,demmel1997applied}. Newton-type approaches for improving eigenpairs have long been known \cite{dongarra1983improving,tisseur2001newton}, and two-sided variants are standard when left eigenvectors are available \cite{hochstenbach2003twosided,freitag2015tuned}. Most of these methods, however, target one or a few eigenpairs. For dense problems in which all eigenpairs are refined simultaneously, a direct extension is usually too expensive. In the real symmetric or Hermitian case, Ogita and Aishima showed that all eigenpairs can be refined in $\mathcal{O}(n^3)$ operations by matrix-multiplication-dominant updates \cite{ogita2018iterative,ogita2019iterative}.

Recent work has broadened this matrix-multiplication-based viewpoint in several directions. For singular value decompositions, analogous refinement frameworks and accelerated variants were developed in \cite{ogita2020svd,uchino2024acceleration}. For real symmetric or Hermitian eigenproblems, the same perspective has also been adapted to selected eigenpairs and to forward-error-oriented eigenvector correction \cite{terao2024subset,terao2026subset,terao2026forward}. On the non-Hermitian side, iterative refinement of Schur decompositions has recently been studied in \cite{bujanovic2023schur}. These advances show that post-processing refinement is effective well beyond the original full symmetric setting, but they do not address diagonalizable non-Hermitian eigendecompositions with simultaneous control of right residuals and left-right biorthogonality. The present paper fills that gap.

The non-Hermitian setting introduces three difficulties that are absent from a straightforward extension of the real symmetric or Hermitian theory. First, once orthogonality is lost, the right-eigenvector correction is driven by $\widehat V^{-1}$, or by an approximation obtained from left eigenvectors. Second, when left and right eigenvectors are updated simultaneously, residual reduction must be balanced against the maintenance of biorthogonality $W^{\mathrm H}V=I$, where $V$ and $W$ denote the exact right and left eigenvector matrices. Third, for clustered eigenvalues, the denominators in the componentwise update formula become small, while the numerically meaningful object is the invariant subspace rather than each individual eigenvector \cite{kato1995perturbation,stewart1990matrix,trefethen2005spectra}.

Matrix-multiplication-based refinement is also attractive computationally. Mixed-precision algorithms and compensated matrix products have developed rapidly \cite{higham2002accuracy,buttari2008exploiting,carson2018accelerating,haidar2018harnessing,higham2022mixed,ozaki2012error,ozaki2025ozaki}, so a refinement strategy whose dominant kernels are matrix products fits naturally into current high-performance and mixed-precision workflows.

The paper has two main messages. First, iterative refinement for diagonalizable non-Hermitian eigendecompositions separates naturally into a right-only regime and a left-right regime, according to the data returned by the initial eigensolver. Second, in the simple-eigenvalue case, the left-right regime admits a local second-order theory for an implementable refinement that uses the available left eigenvectors instead of direct access to $\widehat V^{-1}$. In the sequel, the left-right refinement driven by $Y_W=\widehat W^{\mathrm H}R$ is called the \emph{$W$-method}. Clustered eigenvalues are treated separately as a stabilization extension based on clusterwise re-diagonalization rather than as a completed convergence theory.

The main contributions are as follows.
\begin{itemize}
\item A unified formulation is given for two refinement regimes, right-only and left-right, matching the data returned by the initial eigensolver.
\item In the right-only regime, the update selected by the first-order derivation is proved to satisfy an exact post-update residual identity and a quadratic residual bound.
\item In the left-right regime, the computable driving matrix is proved to be an exact perturbation of the inverse-based one, and the biorthogonality correction is proved to satisfy an exact Newton--Schulz-type error identity. These two facts yield a local second-order estimate for the $W$-method in the simple-eigenvalue case.
\item For clustered eigenvalues, a stabilization mechanism based on clusterwise re-diagonalization and suppression of intracluster corrections is identified and treated as an auxiliary extension beyond the main simple-eigenvalue theory.
\item The numerical study is organized to separate theory-validation experiments for simple eigenvalues from experiments on clustered stabilization, application-derived robustness checks, and mixed-precision cost.
\end{itemize}

The remainder of the paper is organized as follows. Section~\ref{sec:notation} introduces the notation and the exact-arithmetic refinement setting. Section~\ref{sec:simple} develops the simple-eigenvalue theory in the right-only and left-right regimes. Section~\ref{sec:cluster} presents a supplementary stabilization strategy for clustered eigenvalues. Section~\ref{sec:impl} explains the computational realization, Section~\ref{sec:experiments} presents the numerical experiments, and Section~\ref{sec:conclusion} closes the paper.

\section{Problem setting and notation}\label{sec:notation}

We consider a diagonalizable matrix $A\in\mathbb{C}^{n\times n}$ and assume that there exist a right eigenvector matrix $V$, a left eigenvector matrix $W$, and a diagonal eigenvalue matrix $D=\mathrm{diag}(\lambda_1,\dots,\lambda_n)$ such that
\begin{align}
    \label{eq:eig-triple}
    AV = VD, \qquad W^{\mathrm H}A = DW^{\mathrm H}, \qquad W^{\mathrm H}V = I.
\end{align}
The computed approximations are denoted by $\widehat V$, $\widehat W$, and $\widehat D=\mathrm{diag}(\widehat\lambda_1,\dots,\widehat\lambda_n)$. The right residual matrix is
\begin{align}
    \label{eq:right-residual}
    R := A\widehat V - \widehat V\widehat D.
\end{align}
In the left-right regime, the approximate left eigenvector matrix $\widehat W$ is understood to be scaled so that $\widehat W^{\mathrm H}\widehat V\approx I$, hence $\widehat W^{\mathrm H}$ acts as an approximation to $\widehat V^{-1}$. Section~\ref{subsec:biortho} explains how this scaling is imposed numerically before the left-right refinement is applied.
The exact relations in \eqref{eq:eig-triple} are approximated through the residual \eqref{eq:right-residual}. We write the corrections in the form
\begin{align}
    \label{eq:correction-ansatz}
    V = \widehat V(I+E), \qquad W = \widehat W(I+F).
\end{align}
Throughout the paper, first-order approximations are used only to motivate update formulas. Unless the text explicitly says that higher-order terms are neglected, the subsequent propositions and theorems are exact statements about the updates once they have been defined.
Outside Section~\ref{sec:experiments}, all algorithms and statements are interpreted in exact arithmetic. Finite-precision effects are discussed only in the numerical experiments.

\section{Updates for simple eigenvalues}\label{sec:simple}

This section first uses first-order expansions only to motivate the update formulas. Once the updates have been fixed, all propositions and the theorem below are exact algebraic statements; higher-order terms are not discarded unless this is stated explicitly.

\subsection{Right-only case: first-order derivation}\label{subsec:rightheur}

This subsection is heuristic. Assume that the eigenvalues are simple, recall from \eqref{eq:eig-triple} that the exact right eigenvector matrix satisfies $AV=VD$, and write
\begin{align}
    \label{eq:right-heuristic-V}
    V=\widehat V(I+E).
\end{align}
Then
\begin{align}
    \label{eq:right-heuristic-similarity}
    \widehat V^{-1}A\widehat V = (I+E)D(I+E)^{-1}.
\end{align}
To select an update, we discard the higher-order contributions coming from $(I+E)^{-1}-(I-E)$ and from $EDE$. We denote the resulting first-order proxies by $\widetilde D$ and $\widetilde E$, and define the driving matrix
\begin{align}
    \label{eq:right-driving}
    Y := \widehat V^{-1}(A\widehat V-\widehat V\widehat D)=\widehat V^{-1}R.
\end{align}
The formal first-order model is then written without approximation signs as
\begin{align}
    \label{eq:right-heuristic-model}
    Y = (\widetilde D-\widehat D)+\widetilde E\widetilde D-\widetilde D\widetilde E.
\end{align}
Equation \eqref{eq:right-heuristic-model} is not an exact relation for the true correction $E$ in \eqref{eq:right-heuristic-V}; rather, it defines the first-order proxy pair $(\widetilde D,\widetilde E)$ obtained after discarding the higher-order terms. Comparing diagonal and off-diagonal entries in \eqref{eq:right-heuristic-model} gives
\begin{align}
    \label{eq:right-heuristic-shifted}
    \widetilde D := \widehat D+\mathrm{diag}(Y).
\end{align}
The corresponding first-order correction matrix satisfies
\begin{align}
    \label{eq:right-heuristic-correction}
    \widetilde e_{ij} = \frac{y_{ij}}{\widetilde d_j-\widetilde d_i} \quad (i\neq j), \qquad \widetilde e_{ii}=0.
\end{align}
Algorithm~\ref{alg:right} uses the same componentwise formula as \eqref{eq:right-heuristic-correction}, but what is actually computed there is not the true correction $E$ itself. Rather, the algorithm computes the first-order proxy $\widetilde E=[\widetilde e_{ij}]$ selected by the model above. The chain \eqref{eq:right-heuristic-similarity}--\eqref{eq:right-heuristic-correction} serves only to select this update. The next subsection keeps this motivation separate from the formal result: once $\widetilde E$ has been defined by the algorithm, the post-update residual identity and the resulting quadratic estimate are exact.

\subsection{Exact residual identity in the right-only case}\label{subsec:rightexact}

\begin{proposition}[Exact residual identity and quadratic estimate in the right-only case]\label{prop:rightexact}
\emph{Setup.}
Let $A\in\mathbb{C}^{n\times n}$, let $\widehat V\in\mathbb{C}^{n\times n}$ be invertible, and let $\widehat D=\mathrm{diag}(\widehat d_1,\dots,\widehat d_n)$ be diagonal. Assume that Algorithm~\ref{alg:right} is executed in exact arithmetic and that the simple-eigenvalue branch is taken. Define
\begin{align}
    \label{eq:right-prop-defs}
    R:=A\widehat V-\widehat V\widehat D,\qquad Y:=\widehat V^{-1}R,\qquad Y_{\mathrm d}:=\mathrm{diag}(y_{11},\dots,y_{nn}),\qquad \widetilde D:=\widehat D+\mathrm{diag}(Y).
\end{align}
Write $\widetilde D=\mathrm{diag}(\widetilde d_1,\dots,\widetilde d_n)$, and let $\widetilde E=[\widetilde e_{ij}]$ denote the first-order correction matrix computed by the algorithm.

\emph{Claim.}
Because the simple-eigenvalue branch is taken,
\begin{align}
    \label{eq:right-prop-sep}
    \mathrm{sep}(\widetilde D):=\min_{i\neq j}|\widetilde d_i-\widetilde d_j|>0.
\end{align}
The updated residual satisfies the exact identity
\begin{align}
    \label{eq:right-prop-identity}
    \widehat V^{-1}(A\widetilde V-\widetilde V\widetilde D) = Y\widetilde E-Y_{\mathrm d}\widetilde E,
\end{align}
and therefore
\begin{align}
    \label{eq:right-prop-bound}
    \|\widehat V^{-1}(A\widetilde V-\widetilde V\widetilde D)\|_F
    \leq \frac{2\|Y\|_F^2}{\mathrm{sep}(\widetilde D)}.
\end{align}
Thus the updated residual measured in the $\widehat V^{-1}$-scaled norm is of second order in $\|Y\|_F$ once the driving matrix \eqref{eq:right-prop-defs} is small.
\end{proposition}

\begin{proof}
Since $A\widehat V=\widehat V(\widehat D+Y)$, we have
\begin{align}
    \label{eq:right-proof-expand}
    A\widetilde V - \widetilde V\widetilde D = \widehat V\{(\widehat D+Y)(I+\widetilde E)-(I+\widetilde E)\widetilde D\}.
\end{align}
Hence
\begin{align}
    \label{eq:right-proof-rearrange}
    \widehat V^{-1}(A\widetilde V-\widetilde V\widetilde D)=(Y-Y_{\mathrm d})+\widehat D \widetilde E-\widetilde E\widetilde D+Y\widetilde E.
\end{align}
For $i\neq j$,
\begin{align}
    \label{eq:right-proof-entrywise}
    \bigl((Y-Y_{\mathrm d})+\widehat D \widetilde E-\widetilde E\widetilde D\bigr)_{ij}
    = y_{ij}+(\widehat d_i-\widetilde d_j)\widetilde e_{ij}
    = -y_{ii}\widetilde e_{ij},
\end{align}
and the diagonal entries vanish because $\widetilde e_{ii}=0$. Therefore
\begin{align}
    \label{eq:right-proof-identity}
    \widehat V^{-1}(A\widetilde V-\widetilde V\widetilde D) = Y\widetilde E-Y_{\mathrm d}\widetilde E.
\end{align}
Moreover,
\begin{align}
    \label{eq:right-proof-Ebound}
    \|\widetilde E\|_F \le \frac{\|Y\|_F}{\mathrm{sep}(\widetilde D)},
\end{align}
which implies
\begin{align}
    \|Y\widetilde E-Y_{\mathrm d}\widetilde E\|_F
    \le \|Y\widetilde E\|_F+\|Y_{\mathrm d}\widetilde E\|_F
    \le 2\|Y\|_F\|\widetilde E\|_F
    \le \frac{2\|Y\|_F^2}{\mathrm{sep}(\widetilde D)}.
\end{align}
Starting from \eqref{eq:right-proof-expand}, rearranging terms gives \eqref{eq:right-proof-rearrange}. The entrywise cancellation in \eqref{eq:right-proof-entrywise} reduces \eqref{eq:right-proof-rearrange} to the exact residual identity \eqref{eq:right-prop-identity}. Finally, \eqref{eq:right-proof-Ebound} yields the quadratic estimate \eqref{eq:right-prop-bound}.
\end{proof}

In the exact-arithmetic analysis, Algorithm~\ref{alg:right} uses the driving matrix \eqref{eq:right-driving}. Section~\ref{sec:experiments} later compares several finite-precision realizations of this operation.

\begin{algorithm}
\caption{Iterative refinement in the right-only regime (simple eigenvalues)}\label{alg:right}
\begin{algorithmic}[1]
\Require $A\in\mathbb{C}^{n\times n}$, invertible $\widehat V\in\mathbb{C}^{n\times n}$, diagonal $\widehat D=\mathrm{diag}(\widehat d_1,\dots,\widehat d_n)$
\Ensure Updated pair $(\widetilde V,\widetilde D)$ in the simple-eigenvalue regime; if the updated diagonal entries are not sufficiently separated, switch to the cluster treatment of Section~\ref{sec:cluster}
\State $R \gets A\widehat V-\widehat V\widehat D$
\State $Y \gets \widehat V^{-1}R$
\State $\widetilde d_i \gets \widehat d_i+y_{ii}$ for $i=1,\dots,n$
\If{$\min_{i\neq j}|\widetilde d_i-\widetilde d_j|$ is below the chosen threshold}
\State Switch to the cluster treatment of Section~\ref{sec:cluster}
\EndIf
\For{$i=1,\dots,n$}
\For{$j=1,\dots,n$}
\If{$i=j$}
\State $\widetilde e_{ij}\gets 0$
\Else
\State $\widetilde e_{ij}\gets y_{ij}/(\widetilde d_j-\widetilde d_i)$
\EndIf
\EndFor
\EndFor
\State $\widetilde E\gets [\widetilde e_{ij}]$
\State $\widetilde V\gets \widehat V(I+\widetilde E)$
\State $\widetilde D\gets \mathrm{diag}(\widetilde d_1,\dots,\widetilde d_n)$
\State \Return $(\widetilde V,\widetilde D)$
\end{algorithmic}
\end{algorithm}

\subsection{Exact perturbation formula for the driving matrix}\label{subsec:driving}

When both $\widehat V$ and $\widehat W$ are available, the left-right refinement uses the implementable driving matrix
\begin{align}
    \label{eq:yw-def}
    Y_W := \widehat W^{\mathrm H}(A\widehat V-\widehat V\widehat D)
\end{align}
in place of the ideal but usually unavailable quantity
\begin{align}
    \label{eq:ystar-def}
    Y_\star := \widehat V^{-1}(A\widehat V-\widehat V\widehat D).
\end{align}
For brevity, we call Algorithm~\ref{alg:both} the \emph{$W$-method}, since it builds the right correction from the available left eigenvector matrix through \eqref{eq:yw-def}. Two issues must then be analyzed separately: how well $\widehat W^{\mathrm H}$ approximates $\widehat V^{-1}$ in the passage from \eqref{eq:ystar-def} to \eqref{eq:yw-def}, and how the left update should be chosen to maintain biorthogonality.

\begin{proposition}[Exact perturbation representation of $Y_W$]\label{prop:driving}
\emph{Setup.}
Let $A\in\mathbb{C}^{n\times n}$, let $\widehat V\in\mathbb{C}^{n\times n}$ be invertible, let $\widehat W\in\mathbb{C}^{n\times n}$ be arbitrary, and let $\widehat D$ be diagonal. Define
\begin{align}
    \label{eq:driving-prop-defs}
    R:=A\widehat V-\widehat V\widehat D,\qquad
    \Delta:=\widehat W^{\mathrm H}\widehat V-I,\qquad
    Y_\star:=\widehat V^{-1}R,\qquad
    Y_W:=\widehat W^{\mathrm H}R.
\end{align}

\emph{Claim.}
Then
\begin{align}
    \label{eq:driving-identity}
    Y_W-Y_\star=\Delta Y_\star
\end{align}
holds exactly. Hence, for any subordinate norm,
\begin{align}
    \label{eq:driving-bound}
    \|Y_W-Y_\star\|\le \|\Delta\|\,\|Y_\star\|.
\end{align}
Therefore, the error in the implementable driving matrix is first order in the biorthogonality error $\Delta$.
\end{proposition}

\begin{proof}
\begin{align}
    Y_W-Y_\star
    &= (\widehat W^{\mathrm H}-\widehat V^{-1})R \\
    &= (\widehat W^{\mathrm H}\widehat V-I)\widehat V^{-1}R \\
    &= \Delta Y_\star.
\end{align}
The exact identity \eqref{eq:driving-identity} follows immediately, and \eqref{eq:driving-bound} is its norm estimate.
\end{proof}

\subsection{Exact update for the biorthogonality error}\label{subsec:newton}

At the heuristic level, if $Z:=\widehat W^{\mathrm H}\widehat V$ and $\widetilde E=0$, then one Newton--Schulz step for $Z^{-1}$ starting from $X=I$ gives $X_1=2I-Z$. Hence $F_0^{\mathrm H}=I-Z$ corresponds to a single Newton--Schulz-type correction. The next proposition records the exact algebraic identity obtained when this idea is combined with a nonzero right update $\widetilde E$.

\begin{proposition}[Exact update formula for the biorthogonality error]\label{prop:newton}
\emph{Setup.}
Let $\widehat V\in\mathbb{C}^{n\times n}$ be invertible and let $\widehat W\in\mathbb{C}^{n\times n}$ be arbitrary. Define
\begin{align}
    \label{eq:newton-defect}
    \Delta:=\widehat W^{\mathrm H}\widehat V-I.
\end{align}
For any right update matrix $\widetilde E\in\mathbb{C}^{n\times n}$, define the left update by
\begin{align}
    \label{eq:newton-left}
    F^{\mathrm H}:=I-\widehat W^{\mathrm H}\widehat V-\widetilde E = -\Delta-\widetilde E,
\end{align}
and set
\begin{align}
    \label{eq:newton-updates}
    \widetilde V := \widehat V(I+\widetilde E),\qquad \widetilde W := \widehat W(I+F).
\end{align}
Then the updated biorthogonality error
\begin{align}
    \label{eq:newton-updated-defect}
    \widetilde\Delta := \widetilde W^{\mathrm H}\widetilde V-I
\end{align}

\emph{Claim.}
The updated biorthogonality error \eqref{eq:newton-updated-defect} satisfies the exact identity
\begin{align}
    \label{eq:newton-identity}
    \widetilde\Delta = -\widetilde E^2-\widetilde E\Delta-\Delta^2-\Delta^2\widetilde E-\widetilde E\Delta\widetilde E.
\end{align}
Consequently,
\begin{align}
    \label{eq:newton-bound}
    \|\widetilde\Delta\| \le (\|\widetilde E\|+\|\Delta\|)^2+(\|\widetilde E\|+\|\Delta\|)^3.
\end{align}
In particular, if $\widetilde E=0$, then
\begin{align}
    \label{eq:newton-special}
    \widetilde W^{\mathrm H}\widehat V = I-\Delta^2,
\end{align}
so the left update \eqref{eq:newton-left} acts as a Newton--Schulz-type quadratic correction for the biorthogonality error.
\end{proposition}

\begin{proof}
If $\widetilde E=0$, then $(I+F)^{\mathrm H}=I-\Delta=2I-\widehat W^{\mathrm H}\widehat V$, and therefore
\begin{align}
    \widetilde W^{\mathrm H}\widehat V
    = (I+F)^{\mathrm H}\widehat W^{\mathrm H}\widehat V
    = (2I-\widehat W^{\mathrm H}\widehat V)\widehat W^{\mathrm H}\widehat V
    = I-\Delta^2.
\end{align}
For a general $\widetilde E$, we have $(I+F)^{\mathrm H}=I-\Delta-\widetilde E$, hence
\begin{align}
    \widetilde W^{\mathrm H}\widetilde V
    &= (I+F)^{\mathrm H}\widehat W^{\mathrm H}\widehat V(I+\widetilde E) \\
    &= (I-\Delta-\widetilde E)(I+\Delta)(I+\widetilde E) \\
    &= I-\widetilde E^2-\widetilde E\Delta-\Delta^2-\Delta^2\widetilde E-\widetilde E\Delta\widetilde E.
\end{align}
This yields the exact error update \eqref{eq:newton-identity}, and the norm bound \eqref{eq:newton-bound} follows from the triangle inequality and submultiplicativity. The case $\widetilde E=0$ gives \eqref{eq:newton-special}.
\end{proof}

\subsection{Local estimate for the $W$-method}\label{subsec:wlocal}

The following theorem combines the exact perturbation formula \eqref{eq:driving-identity} with the exact error update \eqref{eq:newton-identity}. Locality enters through the separation condition \eqref{eq:wlocal-locality}. No truncation of higher-order terms is used.

\begin{theorem}[Local quadratic-type estimate for the $W$-method in the simple-eigenvalue case]\label{thm:wlocal}
\emph{Setup.}
Let $A\in\mathbb{C}^{n\times n}$, let $\widehat V\in\mathbb{C}^{n\times n}$ be invertible, let $\widehat W\in\mathbb{C}^{n\times n}$ be arbitrary, and let $\widehat D$ be diagonal. Define
\begin{align}
    \label{eq:wlocal-data}
    R:=A\widehat V-\widehat V\widehat D,\qquad
    \Delta:=\widehat W^{\mathrm H}\widehat V-I,\qquad
    Y_\star:=\widehat V^{-1}R,\qquad
    Y_W:=\widehat W^{\mathrm H}R.
\end{align}
Set
\begin{align}
    \label{eq:wlocal-shifts}
    \widetilde D_\star := \widehat D+\mathrm{diag}(Y_\star),\qquad
    \widetilde D := \widehat D+\mathrm{diag}(Y_W),
\end{align}
and write
\begin{align}
    \widetilde D_\star=\mathrm{diag}(\widetilde d_{\star,1},\dots,\widetilde d_{\star,n}),\qquad
    \widetilde D=\mathrm{diag}(\widetilde d_1,\dots,\widetilde d_n).
\end{align}
Define
\begin{align}
    \label{eq:wlocal-locality}
    \gamma:=\mathrm{sep}(\widetilde D_\star)>0,\qquad
    \eta:=\|Y_W-Y_\star\|_F.
\end{align}
Assume the local condition $\eta\le \gamma/4$. Let $\widetilde E_\star=[\widetilde e_{ij}^{(\star)}]$ be defined by
\begin{align}
    \label{eq:wlocal-estar}
    \widetilde e_{ii}^{(\star)} = 0,\qquad
    \widetilde e_{ij}^{(\star)} = \frac{(Y_\star)_{ij}}{\widetilde d_{\star,j}-\widetilde d_{\star,i}},
\end{align}
for $i\neq j$. Finally, let $(\widetilde V,\widetilde W,\widetilde D)$ denote the output of Algorithm~\ref{alg:both} executed in exact arithmetic and in the simple-eigenvalue branch, and let $\widetilde E_W=[\widetilde e_{ij}^{(W)}]$ denote the first-order right-update matrix computed by the algorithm, that is,
\begin{align}
    \label{eq:wlocal-ew}
    \widetilde e_{ii}^{(W)} = 0,\qquad
    \widetilde e_{ij}^{(W)} = \frac{(Y_W)_{ij}}{\widetilde d_j-\widetilde d_i}
\end{align}
for $i\neq j$.

\emph{Claims.}
Then:
\begin{enumerate}
\item $\mathrm{sep}(\widetilde D)\ge \gamma/2$.
\item
\begin{align}
    \label{eq:wlocal-ediff}
    \|\widetilde E_W-\widetilde E_\star\|_F
    \le
    \left(\frac{2}{\gamma}+\frac{4\|Y_\star\|_F}{\gamma^2}\right)\eta.
\end{align}
\item
\begin{align}
    \label{eq:wlocal-residual}
    \|\widehat V^{-1}(A\widetilde V-\widetilde V\widetilde D)\|_F
    \le
    \eta+\frac{2(\|Y_\star\|_F+\eta)^2}{\gamma}.
\end{align}
\item
\begin{align}
    \label{eq:wlocal-defect}
    \|\widetilde W^{\mathrm H}\widetilde V-I\|
    \le
    (\|\widetilde E_W\|+\|\Delta\|)^2+(\|\widetilde E_W\|+\|\Delta\|)^3.
\end{align}
\end{enumerate}
\emph{Interpretation.}
If $\|\Delta\|_2=\mathcal{O}(\|Y_\star\|_F)$, then both the updated right residual in \eqref{eq:wlocal-residual} and the updated biorthogonality error in \eqref{eq:wlocal-defect} are $\mathcal{O}(\|Y_\star\|_F^2)$.
\end{theorem}

\begin{proof}
By Proposition~\ref{prop:driving},
\begin{align}
    \label{eq:wlocal-eta-bound}
    \eta = \|Y_W-Y_\star\|_F \le \|\Delta\|_2\|Y_\star\|_F.
\end{align}
For each diagonal entry,
\begin{align}
    \label{eq:wlocal-diagdiff}
    |\widetilde d_i-\widetilde d_{\star,i}|
    = |(Y_W-Y_\star)_{ii}| \le \eta.
\end{align}
Hence
\begin{align}
    \label{eq:wlocal-sepstep}
    |\widetilde d_j-\widetilde d_i|
    \ge |\widetilde d_{\star,j}-\widetilde d_{\star,i}|-2\eta
    \ge \gamma-2\eta
    \ge \frac{\gamma}{2},
\end{align}
which proves (1).

For $i\neq j$,
\begin{align}
    \label{eq:wlocal-entrydiff}
    \widetilde e_{ij}^{(W)}-\widetilde e_{ij}^{(\star)}
    &=
    \frac{(Y_W-Y_\star)_{ij}}{\widetilde d_j-\widetilde d_i}
    + (Y_\star)_{ij}
    \left(
        \frac{1}{\widetilde d_j-\widetilde d_i}
        - \frac{1}{\widetilde d_{\star,j}-\widetilde d_{\star,i}}
    \right).
\end{align}
Using (1) and
\begin{align}
    \label{eq:wlocal-gapdiff}
    \bigl| \{\widetilde d_j-\widetilde d_i\}
    - \{\widetilde d_{\star,j}-\widetilde d_{\star,i}\}\bigr|
    \le 2\eta,
\end{align}
we obtain
\begin{align}
    \label{eq:wlocal-entrybound}
    |\widetilde e_{ij}^{(W)}-\widetilde e_{ij}^{(\star)}|
    \le \frac{2}{\gamma}|(Y_W-Y_\star)_{ij}| + \frac{4\eta}{\gamma^2}|(Y_\star)_{ij}|,
\end{align}
and squaring and summing gives (2).

Let
\begin{align}
    \label{eq:wlocal-S}
    S := \widehat V^{-1}(A\widetilde V-\widetilde V\widetilde D).
\end{align}
Since $A\widehat V=\widehat V(\widehat D+Y_\star)$ and the algorithm sets $\widetilde D=\widehat D+\mathrm{diag}(Y_W)$,
\begin{align}
    \label{eq:wlocal-Sexpand}
    S = Y_\star - \mathrm{diag}(Y_W) + \widehat D \widetilde E_W - \widetilde E_W\widetilde D + Y_\star \widetilde E_W.
\end{align}
From the definition of $\widetilde E_W$,
\begin{align}
    \label{eq:wlocal-cancel}
    \widehat D \widetilde E_W-\widetilde E_W\widetilde D
    = -\{Y_W-\mathrm{diag}(Y_W)\}-\mathrm{diag}(Y_W)\widetilde E_W.
\end{align}
Hence
\begin{align}
    \label{eq:wlocal-Sreduced}
    S = (Y_\star-Y_W)+\{Y_\star-\mathrm{diag}(Y_W)\}\widetilde E_W.
\end{align}
Therefore,
\begin{align}
    \label{eq:wlocal-Sbound}
    \|S\|_F
    \le \eta + \|Y_\star-\mathrm{diag}(Y_W)\|_F\|\widetilde E_W\|_F.
\end{align}
Also,
\begin{align}
    \label{eq:wlocal-diagbound}
    \|Y_\star-\mathrm{diag}(Y_W)\|_F
    \le \|Y_\star\|_F+\eta,
\end{align}
and by (1),
\begin{align}
    \label{eq:wlocal-EWbound}
    \|\widetilde E_W\|_F
    \le \frac{\|Y_W\|_F}{\mathrm{sep}(\widetilde D)}
    \le \frac{2(\|Y_\star\|_F+\eta)}{\gamma}.
\end{align}
The diagonal perturbation estimate \eqref{eq:wlocal-diagdiff} yields the separation bound \eqref{eq:wlocal-sepstep}, which proves item 1. Next, \eqref{eq:wlocal-entrydiff} together with \eqref{eq:wlocal-gapdiff} gives the entrywise estimate \eqref{eq:wlocal-entrybound}, and hence \eqref{eq:wlocal-ediff}. For the residual, we introduce \eqref{eq:wlocal-S}, expand it by \eqref{eq:wlocal-Sexpand}, reduce it using \eqref{eq:wlocal-cancel} to \eqref{eq:wlocal-Sreduced}, and finally combine \eqref{eq:wlocal-Sbound}, \eqref{eq:wlocal-diagbound}, and \eqref{eq:wlocal-EWbound} to obtain \eqref{eq:wlocal-residual}. Finally, \eqref{eq:wlocal-defect} follows directly from Proposition~\ref{prop:newton} with $\widetilde E=\widetilde E_W$. The final second-order conclusion follows by substituting \eqref{eq:wlocal-eta-bound} and $\|\Delta\|_2=\mathcal{O}(\|Y_\star\|_F)$ into \eqref{eq:wlocal-residual} and \eqref{eq:wlocal-defect}.
\end{proof}

\begin{remark}
Theorem~\ref{thm:wlocal} is a local result for the refinement setting considered in this paper. It assumes that the input eigendecomposition is already sufficiently accurate to satisfy the separation condition and a small-biorthogonality-error condition, and it does not analyze the size of the attraction region. It also does not cover clustered eigenvalues.
\end{remark}

\begin{algorithm}
\caption{Iterative refinement in the left-right regime (simple eigenvalues)}\label{alg:both}
\begin{algorithmic}[1]
\Require $A\in\mathbb{C}^{n\times n}$, invertible $\widehat V\in\mathbb{C}^{n\times n}$, $\widehat W\in\mathbb{C}^{n\times n}$, diagonal $\widehat D=\mathrm{diag}(\widehat d_1,\dots,\widehat d_n)$
\Ensure Updated triplet $(\widetilde V,\widetilde W,\widetilde D)$ in the simple-eigenvalue regime; if the updated diagonal entries are not sufficiently separated, switch to Section~\ref{sec:cluster}
\State $R\gets A\widehat V-\widehat V\widehat D$
\State $Y\gets \widehat W^{\mathrm H}R$
\State $\widetilde d_i\gets \widehat d_i+y_{ii}$ for $i=1,\dots,n$
\If{$\min_{i\neq j}|\widetilde d_i-\widetilde d_j|$ is below the chosen threshold}
\State Switch to Section~\ref{sec:cluster}
\EndIf
\For{$i=1,\dots,n$}
\For{$j=1,\dots,n$}
\If{$i=j$}
\State $\widetilde e_{ij}\gets 0$
\Else
\State $\widetilde e_{ij}\gets y_{ij}/(\widetilde d_j-\widetilde d_i)$
\EndIf
\EndFor
\EndFor
\State $\widetilde E\gets [\widetilde e_{ij}]$
\State $F^{\mathrm H}\gets I-\widehat W^{\mathrm H}\widehat V-\widetilde E$
\State $\widetilde V\gets \widehat V(I+\widetilde E)$
\State $\widetilde W\gets \widehat W(I+F)$
\State $\widetilde D\gets \mathrm{diag}(\widetilde d_1,\dots,\widetilde d_n)$
\State \Return $(\widetilde V,\widetilde W,\widetilde D)$
\end{algorithmic}
\end{algorithm}

\section{Stabilization for clustered eigenvalues}\label{sec:cluster}

This section is supplementary to the simple-eigenvalue theory in Section~\ref{sec:simple}. If eigenvalues are close, the denominator in $\widetilde e_{ij}=y_{ij}/(\widetilde d_j-\widetilde d_i)$ becomes small and the componentwise correction may become unstable. The underlying reason is that for a cluster the invariant subspace is numerically more stable than any particular basis of eigenvectors \cite{kato1995perturbation,stewart1990matrix,trefethen2005spectra}.

We therefore separate the intracluster basis adjustment from the intercluster correction. Let the approximate eigenvalues $\{\widehat\lambda_i\}$ be partitioned into clusters by a threshold $\delta>0$. For a cluster $J\subset\{1,\dots,n\}$ we form the projected matrix
\begin{align}
    \label{eq:cluster-projected}
    B_J := \widehat W_J^{\mathrm H}A\widehat V_J\in\mathbb{C}^{k\times k},
\end{align}
where $\widehat V_J$ and $\widehat W_J$ are the submatrices corresponding to the cluster. If $\widehat W^{\mathrm H}\widehat V\approx I$, then $B_J$ approximates the action of $A$ on the corresponding invariant subspace. We diagonalize
\begin{align}
    \label{eq:cluster-diagonalization}
    B_J=S\Theta S^{-1}
\end{align}
and update
\begin{align}
    \label{eq:cluster-update}
    \widehat V_J \leftarrow \widehat V_JS,\qquad
    \widehat W_J \leftarrow \widehat W_JS^{-\mathrm H},\qquad
    \widehat D_J \leftarrow \Theta.
\end{align}
The construction \eqref{eq:cluster-projected}--\eqref{eq:cluster-update} isolates the basis ambiguity inside the approximate invariant subspace before any intercluster correction is attempted.

\begin{proposition}[Clusterwise re-diagonalization preserves the approximate invariant subspace]\label{prop:cluster}
\emph{Setup.}
Let $J$ be a cluster and let
\begin{align}
    \label{eq:cluster-prop-defs}
    B_J:=\widehat W_J^{\mathrm H}A\widehat V_J,\qquad
    B_J=S\Theta S^{-1},
\end{align}
where $S\in\mathbb{C}^{k\times k}$ is invertible and $\Theta$ is diagonal. Define
\begin{align}
    \label{eq:cluster-prop-update}
    \widetilde V_J := \widehat V_JS,\qquad
    \widetilde W_J := \widehat W_JS^{-\mathrm H}.
\end{align}

\emph{Claim.}
Then
\begin{align}
    \label{eq:cluster-prop-spans}
    \mathrm{span}(\widetilde V_J) = \mathrm{span}(\widehat V_J),\qquad
    \mathrm{span}(\widetilde W_J) = \mathrm{span}(\widehat W_J),
\end{align}
and
\begin{align}
    \label{eq:cluster-prop-project}
    \widetilde W_J^{\mathrm H}\widetilde V_J = S^{-1}(\widehat W_J^{\mathrm H}\widehat V_J)S,\qquad
    \widetilde W_J^{\mathrm H}A\widetilde V_J = \Theta.
\end{align}
In particular, if $\widehat W_J^{\mathrm H}\widehat V_J=I$, then the same identity holds after the update.
\end{proposition}

\begin{proof}
Since $S$ is invertible, $\widetilde V_J$ and $\widetilde W_J$ span the same column spaces as $\widehat V_J$ and $\widehat W_J$. Moreover,
\begin{align}
    \widetilde W_J^{\mathrm H}\widetilde V_J = S^{-1}\widehat W_J^{\mathrm H}\widehat V_JS,
\end{align}
and
\begin{align}
    \widetilde W_J^{\mathrm H}A\widetilde V_J
    = S^{-1}\widehat W_J^{\mathrm H}A\widehat V_JS
    = S^{-1}B_JS
    = \Theta.
\end{align}
Hence \eqref{eq:cluster-prop-spans} and \eqref{eq:cluster-prop-project} both follow directly from the basis change \eqref{eq:cluster-prop-update}.
\end{proof}

\begin{remark}
Proposition~\ref{prop:cluster} shows through \eqref{eq:cluster-prop-spans} and \eqref{eq:cluster-prop-project} that clusterwise re-diagonalization changes only the basis inside the approximate invariant subspace and does not change the subspace itself. This is why suppressing direct intracluster corrections and letting the small projected problem absorb the basis ambiguity is numerically coherent. We use this mechanism as a stabilization strategy; a complete convergence theory for the clustered case is left open.
\end{remark}

\section{Computational realization}\label{sec:impl}

\subsection{Use cases}

The choice of refinement method depends on which approximate quantities are returned by the initial eigensolver. In the exact-arithmetic formulation of Section~\ref{sec:simple}, Algorithm~\ref{alg:right} uses the inverse-based driving matrix \eqref{eq:right-driving} when only $\widehat V$ and $\widehat D$ are available. If $\widehat V$, $\widehat W$, and $\widehat D$ are all available, then one first performs a biorthogonalization preprocessing step and then applies Algorithm~\ref{alg:both}, which uses the implementable driving matrix \eqref{eq:yw-def}. In the numerical experiments, these exact operations are realized by finite-precision kernels, while the biorthogonality error $\|\widehat W^{\mathrm H}\widehat V-I\|$ is monitored simultaneously. The corresponding behaviors are shown in Figures~\ref{fig:simple} and \ref{fig:preprocess}.

\subsection{Biorthogonalization preprocessing}\label{subsec:biortho}

The $W$-method assumes $\widehat W^{\mathrm H}\approx \widehat V^{-1}$, equivalently $\widehat W^{\mathrm H}\widehat V\approx I$. In practice, the normalization returned by an eigensolver may not satisfy this relation. If
\begin{align}
    \label{eq:biortho-split}
    \widehat W^{\mathrm H}\widehat V = \Sigma+G = (I+G\Sigma^{-1})\Sigma,
\end{align}
where $\Sigma=\mathrm{diag}(\sigma_1,\dots,\sigma_n)$ and $G$ has zero diagonal, then the update
\begin{align}
    \label{eq:biortho-firstorder}
    \widehat W \leftarrow \widehat W\Sigma^{-\mathrm H}(I-\Sigma^{-\mathrm H}G^{\mathrm H})
\end{align}
is obtained from the first-order approximation
\[
(\Sigma+G)^{-\mathrm H}
=
\Sigma^{-\mathrm H}(I+G\Sigma^{-1})^{-\mathrm H}
=
\Sigma^{-\mathrm H}(I-\Sigma^{-\mathrm H}G^{\mathrm H})+\mathcal{O}(\|G\|^2).
\]
In the numerical experiments below, however, we use the exact biorthogonalization
\begin{align}
    \label{eq:biortho-exact}
    \widehat W \leftarrow \widehat W S^{-\mathrm H},\qquad S:=\widehat W^{\mathrm H}\widehat V,
\end{align}
up to roundoff, in order to make the comparisons cleaner. Thus \eqref{eq:biortho-firstorder} explains the intended correction, whereas \eqref{eq:biortho-exact} is the operation used in the experiments.

\subsection{Finite-precision realizations in the right-only regime}

In the numerical experiments, the exact quantity \eqref{eq:right-driving} from Algorithm~\ref{alg:right} is realized in several ways.
\begin{itemize}
\item \textbf{Direct solve}: solve $\widehat VY=R$ at every iteration. This is the most stable option.
\item \textbf{Updated approximate inverse}: maintain $\widehat M\approx \widehat V^{-1}$ and compute $Y=\widehat M R$. Under the update $\widehat V\leftarrow\widehat V(I+\widetilde E)$, one uses the first-order inverse update
\begin{align}
    \label{eq:inverse-update}
    \widehat M \leftarrow (I-\widetilde E)\widehat M.
\end{align}
Equation \eqref{eq:inverse-update} is the finite-precision analogue of differentiating the inverse map.
\item \textbf{Fixed inverse}: compute $\widehat M$ only once and reuse it during the refinement.
\end{itemize}
All variants have per-iteration complexity $\mathcal{O}(n^3)$, with matrix products as dominant kernels.

\section{Numerical experiments}\label{sec:experiments}

The numerical results are produced by a reproducible C implementation based on LAPACK/BLAS. All experiments were run on an Apple M4 Mac mini with 16\,GB memory under macOS 26.3.1. The code was compiled with Apple clang 17.0.0 and linked against Accelerate; OpenMP support was provided through \texttt{libomp}, and the timing runs used \texttt{OMP\_NUM\_THREADS=2} with \texttt{VECLIB\_MAXIMUM\_THREADS=1}. The experiments serve three distinct purposes. Figures~\ref{fig:simple}--\ref{fig:sensitivity} (left) test the mechanisms predicted by the simple-eigenvalue analysis, which is the main theoretical focus of the paper, especially residual reduction and sensitivity to the biorthogonality error. Figure~\ref{fig:suitesparse} supplements these controlled tests with application-derived matrices. Figures~\ref{fig:cluster}--\ref{fig:timing} document the clustered stabilization extension and the mixed-precision cost profile. Most test matrices are generated rather than drawn from application benchmarks because the purpose is to vary separation, nonnormality, and cluster-basis conditioning independently. This design isolates the mechanisms analyzed in the paper; the SuiteSparse examples are included only as robustness checks outside that controlled setting.

\subsection{Metrics}

We measure the relative residual
\begin{align}
    \label{eq:metric-residual}
    \frac{\|A\widehat V-\widehat V\widehat D\|_F}{\|A\|_F},
\end{align}
the biorthogonality error
\begin{align}
    \label{eq:metric-biorth}
    \|\widehat W^{\mathrm H}\widehat V-I\|_F,
\end{align}
and the eigenvalue consistency error
\begin{align}
    \label{eq:metric-consistency}
    \frac{\|\mathrm{diag}(\widehat W^{\mathrm H}A\widehat V)-\mathrm{diag}(\widehat D)\|_2}{\|A\|_F}.
\end{align}
The residual \eqref{eq:metric-residual} is used throughout. The latter two quantities are mainly used in the left-right regime and in the clustered experiments, where biorthogonality and projected eigenvalue consistency are part of the mechanism under study.

\subsection{Setup and precision model}

For the simple-eigenvalue and $\alpha$-sensitivity experiments, we generate
\begin{align}
    \label{eq:setup-simple-family}
    A = X\mathrm{diag}(\lambda_1,\dots,\lambda_n)X^{-1},\qquad X=I+\alpha N,
\end{align}
where $N_{ij}\sim\mathcal{N}(0,1)$. The parameter $\alpha$ controls the strength of nonnormality. In the real experiments with simple eigenvalues, the diagonal entries are uniformly spaced on $[-1,1]$, namely
\begin{align}
    \label{eq:setup-real-spectrum}
    \lambda_i = -1 + \frac{2(i-1)}{n-1},\qquad i=1,\dots,n.
\end{align}
In the complex experiment, the eigenvalues are placed on the unit circle,
\begin{align}
    \label{eq:setup-complex-spectrum}
    \lambda_i = \exp\!\left(\frac{2\pi\mathrm{i}(i-1)}{n}\right),\qquad i=1,\dots,n.
\end{align}
Unless stated otherwise, the simple-spectrum experiments use $\alpha=0.05$. The initial approximation is computed in single precision (\texttt{sgeev} for real problems and \texttt{cgeev} for complex problems) and then cast to double precision. Residual evaluation, direct solves in the right-only regime, inverse construction, biorthogonalization, and the clusterwise projected eigenproblems are all carried out in double precision. Thus the figures below correspond to a ``single-precision initial solve + double-precision refinement'' workflow, and the \textbf{SP+IR} timings in Figure~\ref{fig:timing} include the full cost of refinement.

For the clustered experiments, we deliberately construct matrices for which the naive componentwise update becomes unstable. The cluster block is generated as
\begin{align}
    \label{eq:setup-cluster-family}
    D &= \mathrm{diag}\!\left(\mu-\frac{k-1}{2}h,\dots,\mu+\frac{k-1}{2}h,\lambda_{k+1},\dots,\lambda_n\right), \\
    X &= \begin{bmatrix} B & 0 \\ 0 & I \end{bmatrix} + \eta G,\qquad B_{ij}=(1+\rho(i-1))^{j-1},
\end{align}
and we set $A=XDX^{-1}$. Here $k=6$, $\mu=0.25$, $h=3\times 10^{-7}$, $\rho=2\times 10^{-3}$, $\eta=10^{-3}$, and $G_{ij}\sim\mathcal{N}(0,1)$. The Vandermonde-type cluster basis $B$ has Frobenius-condition number about $4.0\times 10^{14}$, while the full similarity transform $X$ has condition number about $3.0\times 10^4$. This makes the initial single-precision cluster basis strongly mixed and exposes the instability of the naive update. For the cluster-conditioning study, we vary $\rho$ and report medians over three random seeds.

\subsection{Simple eigenvalues}\label{subsec:simpleexp}

Figure~\ref{fig:simple} uses the real simple-spectrum model \eqref{eq:setup-simple-family} with the spectrum \eqref{eq:setup-real-spectrum}, dimension $n=200$, parameter $\alpha=0.05$, and five refinement steps. We compare three concrete realizations: Algorithm~\ref{alg:both} in the left-right regime ($W$-method), the right-only algorithm realized with a fixed approximate inverse, and the right-only algorithm realized with a direct linear solve at each step. All three drive the residual \eqref{eq:metric-residual} to the double-precision limit within a few iterations, which is consistent with the simple-eigenvalue theory.

\begin{figure}[t]
\centering
\begin{tikzpicture}
\begin{axis}[
    width=0.92\linewidth,
    height=0.46\linewidth,
    ymode=log,
    xlabel={Iteration},
    ylabel={Relative residual},
    legend style={at={(0.98,0.98)},anchor=north east},
    xtick={0,1,2,3,4,5},
    grid=both]
\addlegendimage{blue, mark=*}
\addlegendentry{$W$-method}
\addlegendimage{red, mark=square*}
\addlegendentry{fixed inverse}
\addlegendimage{brown!70!black, mark=triangle, mark size=4pt}
\addlegendentry{direct solve}
\addplot+[brown!70!black, mark=triangle, mark size=4pt, forget plot]
table[col sep=space, comment chars={@}, x index=0, y index=3] {data/simple_convergence.dat};
\addplot+[red, mark=square*, forget plot]
table[col sep=space, comment chars={@}, x index=0, y index=2] {data/simple_convergence.dat};
\addplot+[blue, mark=*, forget plot]
table[col sep=space, comment chars={@}, x index=0, y index=1] {data/simple_convergence.dat};
\end{axis}
\end{tikzpicture}
\caption{Residual histories for the left-right $W$-method and for two right-only realizations of Algorithm~\ref{alg:right}.}
\label{fig:simple}
\end{figure}
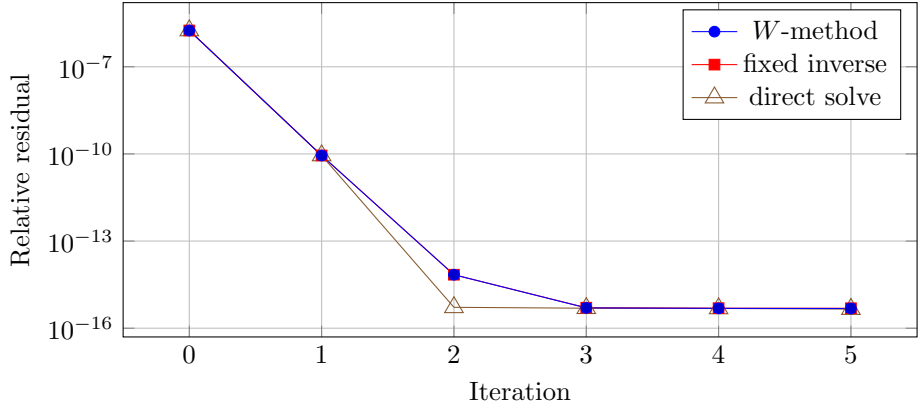

The same experiment also records the biorthogonality error \eqref{eq:metric-biorth} and the eigenvalue consistency error \eqref{eq:metric-consistency}. For the $W$-method, the biorthogonality error can increase temporarily at the first iteration because the left update first removes the leading residual component; after that, it returns to the double-precision level.

\subsection{Complex eigenvalues}

The derivation is written for complex matrices and applies directly to diagonalizable problems with complex eigenvalues. Figure~\ref{fig:complex} applies Algorithm~\ref{alg:both} to the family \eqref{eq:setup-simple-family} together with the unit-circle spectrum \eqref{eq:setup-complex-spectrum}, with dimension $n=120$, similarity amplitude $\alpha=0.05$, and five refinement steps. The residual \eqref{eq:metric-residual} decreases rapidly, and the biorthogonality error \eqref{eq:metric-biorth} again shows a transient increase followed by convergence to the rounding-error level.

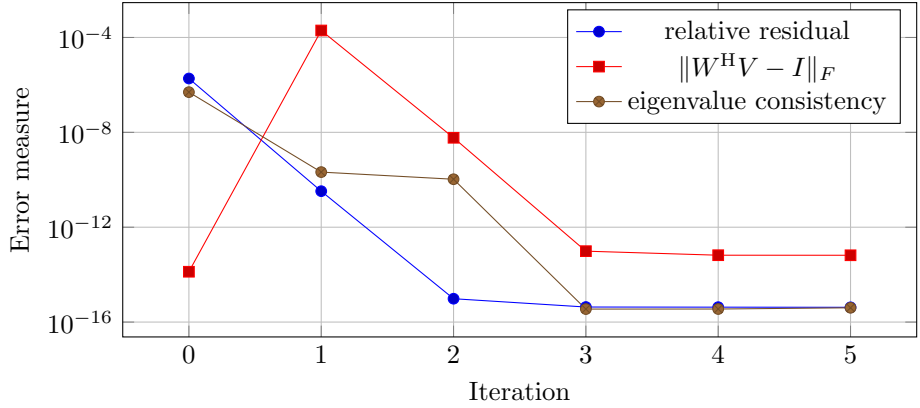
\begin{figure}[t]
\centering
\begin{tikzpicture}
\begin{axis}[
    width=0.92\linewidth,
    height=0.46\linewidth,
    ymode=log,
    xlabel={Iteration},
    ylabel={Error measure},
    legend style={at={(0.98,0.98)},anchor=north east},
    xtick={0,1,2,3,4,5},
    grid=both]
\addplot table[col sep=space, comment chars={@}, x index=0, y index=1] {data/complex_convergence.dat};
\addlegendentry{relative residual}
\addplot table[col sep=space, comment chars={@}, x index=0, y index=2] {data/complex_convergence.dat};
\addlegendentry{$\|W^{\mathrm H}V-I\|_F$}
\addplot table[col sep=space, comment chars={@}, x index=0, y index=3] {data/complex_convergence.dat};
\addlegendentry{eigenvalue consistency}
\end{axis}
\end{tikzpicture}
\caption{Refinement history for Algorithm~\ref{alg:both} on a diagonalizable matrix with complex eigenvalues.}
\label{fig:complex}
\end{figure}

\subsection{Effect of biorthogonalization preprocessing}

Figure~\ref{fig:preprocess} uses the same real test family \eqref{eq:setup-simple-family}--\eqref{eq:setup-real-spectrum} as Figure~\ref{fig:simple}, again with $n=200$, $\alpha=0.05$, and five refinement steps. Both curves use Algorithm~\ref{alg:both}; the only difference is whether the biorthogonalization preprocessing \eqref{eq:biortho-exact} is applied before refinement starts. Without preprocessing, refinement still improves the decomposition, but the early decrease in \eqref{eq:metric-residual} is slower. With preprocessing, the biorthogonality error \eqref{eq:metric-biorth} is small from the start and the convergence becomes much sharper.

\begin{figure}[t]
\centering
\begin{tikzpicture}
\begin{axis}[
    width=0.92\linewidth,
    height=0.46\linewidth,
    ymode=log,
    xlabel={Iteration},
    ylabel={Relative residual},
    legend style={at={(0.02,0.02)},anchor=south west},
    xtick={0,1,2,3,4,5},
    grid=both]
\addplot table[col sep=space, comment chars={@}, x index=0, y index=1] {data/preprocess_effect.dat};
\addlegendentry{without preprocessing}
\addplot table[col sep=space, comment chars={@}, x index=0, y index=2] {data/preprocess_effect.dat};
\addlegendentry{with biorthogonalization}
\end{axis}
\end{tikzpicture}
\caption{Effect of the biorthogonalization preprocessing on Algorithm~\ref{alg:both}.}
\label{fig:preprocess}
\end{figure}
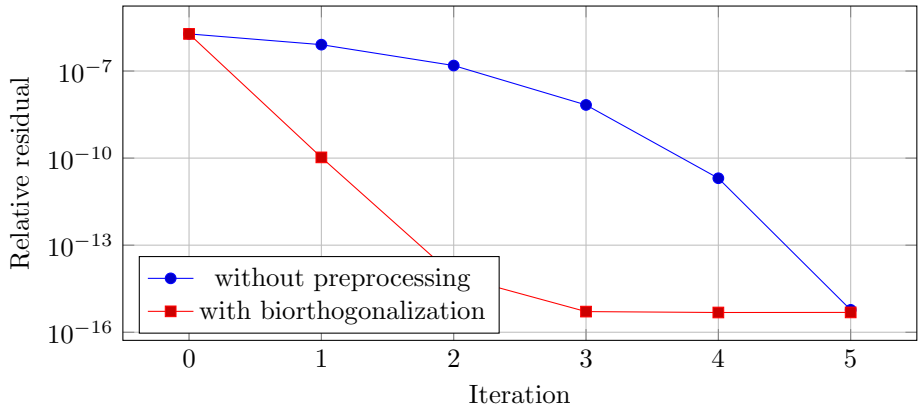

\subsection{Clustered eigenvalues}

Figure~\ref{fig:cluster} uses the clustered model \eqref{eq:setup-cluster-family} with $n=160$, cluster size $k=6$, center $\mu=0.25$, gap $h=3\times 10^{-7}$, basis parameter $\rho=2\times 10^{-3}$, coupling amplitude $\eta=10^{-3}$, and cluster threshold $\delta=10^{-6}$. We run eight refinement steps and report the residual \eqref{eq:metric-residual}. Both curves are left-right refinements: the naive variant applies the componentwise update without cluster detection, whereas the cluster-aware variant adds the projected re-diagonalization \eqref{eq:cluster-diagonalization}--\eqref{eq:cluster-update} and suppresses the direct intracluster corrections. To make the comparison fair, both variants are followed by the same exact biorthogonalization after each iteration. Even under this fair comparison, the naive residual stagnates and eventually grows from $2.1\times 10^{-6}$ to $4.9\times 10^{-1}$ after eight iterations, whereas the cluster-aware method reaches the double-precision level in about four iterations. This indicates that the instability comes from the intracluster componentwise correction itself, not merely from a loss of biorthogonality.

\begin{figure}[t]
\centering
\begin{tikzpicture}
\begin{axis}[
    width=0.92\linewidth,
    height=0.46\linewidth,
    ymode=log,
    xlabel={Iteration},
    ylabel={Relative residual},
    legend style={at={(0.02,0.02)},anchor=south west},
    xtick={0,1,2,3,4,5,6,7,8},
    grid=both]
\addplot table[col sep=space, comment chars={@}, x index=0, y index=1] {data/cluster_handling.dat};
\addlegendentry{naive}
\addplot table[col sep=space, comment chars={@}, x index=0, y index=2] {data/cluster_handling.dat};
\addlegendentry{cluster-aware}
\end{axis}
\end{tikzpicture}
\caption{Comparison of two left-right refinements on the clustered test: the naive componentwise update and the cluster-aware update. Both use the same exact biorthogonalization after each iteration.}
\label{fig:cluster}
\end{figure}
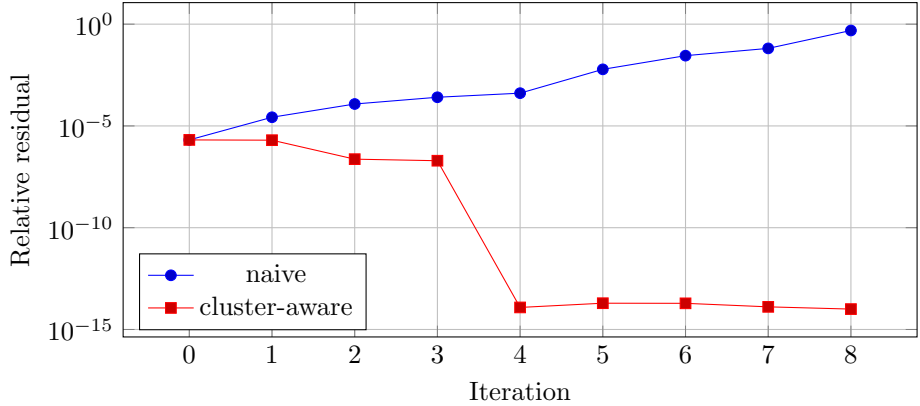

\subsection{Sensitivity to the conditioning of the cluster basis}

Figure~\ref{fig:clustercond} uses the same clustered model \eqref{eq:setup-cluster-family} with $n=160$, $k=6$, $h=3\times 10^{-7}$, $\eta=10^{-3}$, and $\delta=10^{-6}$, but now varies the basis parameter $\rho\in\{5\times 10^{-4},10^{-3},2\times 10^{-3},5\times 10^{-3},10^{-2}\}$. The horizontal axis is the Frobenius-condition number of $B$, and the vertical axis is the residual \eqref{eq:metric-residual} after four iterations. Each point is the median over three random seeds. Both curves again use the left-right refinement, with or without the cluster-aware modification. The naive method is highly sensitive to the conditioning of the cluster basis, while the cluster-aware method remains close to $10^{-14}$ over a wide range. This confirms that the previous experiment is not an isolated extreme example.

\begin{figure}[t]
\centering
\begin{tikzpicture}
\begin{axis}[
    width=0.92\linewidth,
    height=0.46\linewidth,
    xmode=log,
    ymode=log,
    xlabel={$\kappa_F(B)$},
    ylabel={Residual after 4 iterations},
    legend style={at={(0.5,-0.24)},anchor=north,legend columns=2,font=\scriptsize},
    grid=both]
\addplot table[col sep=space, comment chars={@}, x index=1, y index=2] {data/cluster_condition_sensitivity.dat};
\addlegendentry{naive}
\addplot table[col sep=space, comment chars={@}, x index=1, y index=3] {data/cluster_condition_sensitivity.dat};
\addlegendentry{cluster-aware}
\end{axis}
\end{tikzpicture}
\caption{Sensitivity of the naive and cluster-aware updates to the conditioning of the cluster basis.}
\label{fig:clustercond}
\end{figure}
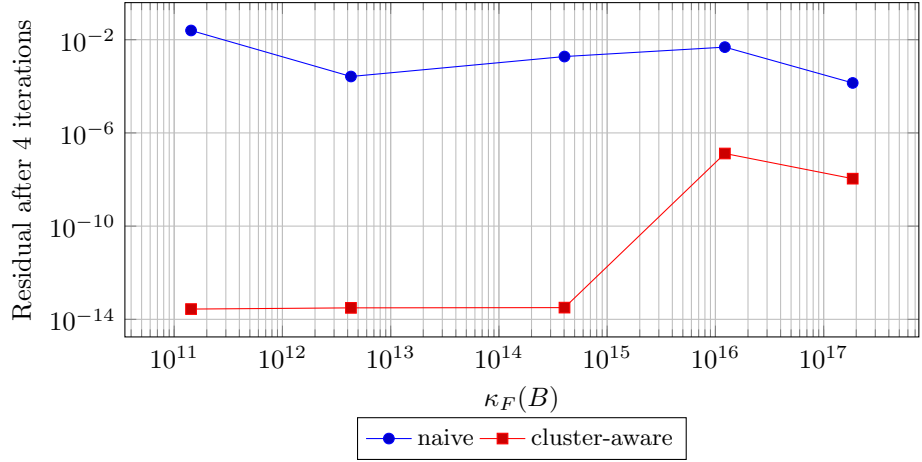

\subsection{Parameter sensitivity}

Figure~\ref{fig:sensitivity} summarizes two parameter studies. On the left, we return to the real simple-spectrum family with $n=200$ and vary $\alpha$ according to
\begin{align}
    \label{eq:alpha-grid}
    \alpha\in\{10^{-2},\,3\times 10^{-2},\,5\times 10^{-2},\,10^{-1},\,2\times 10^{-1}\},
\end{align}
reporting the residual \eqref{eq:metric-residual} after five refinement steps. The two methods in the left panel are the left-right $W$-method and the right-only direct-solve realization of Algorithm~\ref{alg:right}. Across the tested grid \eqref{eq:alpha-grid}, both methods reach the double-precision level after five steps. The final residual therefore remains small even when $\alpha$ increases to $0.2$, although the biorthogonality error of the $W$-method grows with $\alpha$ in the background. On the right, we use the clustered matrix \eqref{eq:setup-cluster-family} with $n=160$ and vary $\delta$ according to
\begin{align}
    \label{eq:delta-grid}
    \delta\in\{3\times 10^{-8},\,10^{-7},\,2\times 10^{-7},\,5\times 10^{-7},\,10^{-6},\,3\times 10^{-6}\},
\end{align}
again reporting the residual \eqref{eq:metric-residual} after four refinement steps. The parameter grids are therefore given explicitly by \eqref{eq:alpha-grid} and \eqref{eq:delta-grid}. If $\delta$ is too small, the full cluster is not detected and the cluster-aware method behaves essentially like the naive method. Once $\delta$ is large enough to capture the whole cluster, stable convergence is recovered.

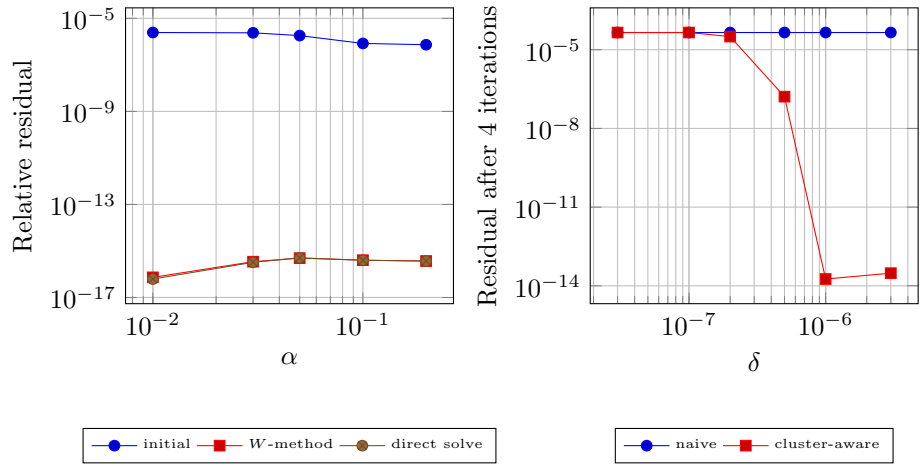
\begin{figure}[t]
\centering
\begin{tikzpicture}
\begin{groupplot}[
    group style={group size=2 by 1, horizontal sep=1.8cm},
    width=0.33\linewidth,
    height=0.30\linewidth,
    scale only axis,
    xmode=log,
    ymode=log]
\nextgroupplot[
    xlabel={$\alpha$},
    ylabel={Relative residual},
    legend style={at={(0.5,-0.42)},anchor=north,font=\tiny},
    legend columns=3,
    grid=both]
\addplot table[col sep=space, comment chars={@}, x index=0, y index=1] {data/alpha_sensitivity.dat};
\addlegendentry{initial}
\addplot table[col sep=space, comment chars={@}, x index=0, y index=2] {data/alpha_sensitivity.dat};
\addlegendentry{$W$-method}
\addplot table[col sep=space, comment chars={@}, x index=0, y index=3] {data/alpha_sensitivity.dat};
\addlegendentry{direct solve}

\nextgroupplot[
    xmode=log,
    ymode=log,
    xlabel={$\delta$},
    ylabel={Residual after 4 iterations},
    legend style={at={(0.5,-0.42)},anchor=north,legend columns=2,font=\tiny},
    grid=both]
\addplot table[col sep=space, comment chars={@}, x index=0, y index=1] {data/delta_sensitivity.dat};
\addlegendentry{naive}
\addplot table[col sep=space, comment chars={@}, x index=0, y index=2] {data/delta_sensitivity.dat};
\addlegendentry{cluster-aware}
\end{groupplot}
\end{tikzpicture}
\caption{Left: sensitivity to the nonnormality parameter $\alpha$. Right: sensitivity to the cluster threshold $\delta$.}
\label{fig:sensitivity}
\end{figure}

\subsection{Application-derived dense benchmarks}

To complement the generated matrices above, we tested dense copies of application-derived matrices from the SuiteSparse Matrix Collection \cite{davis2011university}. The candidate set consisted of Grund/\texttt{d\_ss}, HB/\texttt{fs\_183\_6}, HB/\texttt{fs\_541\_2}, HB/\texttt{shl\_200}, and HB/\texttt{bp\_400}. Since the present paper studies post-processing of an already accurate eigendecomposition, we retained only matrices for which a dense double-precision eigensolve of the complexified matrix produced the relative residual \eqref{eq:metric-residual} below $10^{-8}$. This left four matrices: \texttt{d\_ss} $(n=53,\ \mathrm{nnz}=149)$, \texttt{fs\_541\_2} $(n=541,\ \mathrm{nnz}=4285)$, \texttt{shl\_200} $(n=663,\ \mathrm{nnz}=1726)$, and \texttt{bp\_400} $(n=822,\ \mathrm{nnz}=4028)$. Each real matrix is converted to a dense complex matrix so that complex eigenpairs are treated uniformly. These tests are not intended to assess sparse scalability; they are robustness checks for the dense post-processing procedure studied here.

For each retained benchmark, the right-only refinement is started directly from the single-precision eigendecomposition, that is, without additional preprocessing. The left-right benchmark starts with one exact biorthogonalization \eqref{eq:biortho-exact}. If the initial approximate eigenvalues contain a cluster with threshold $\delta_{\mathrm{bench}}=10^{-5}$, then one clusterwise re-diagonalization preprocessing step is applied before the first $W$-iteration; on the present benchmark set this additional preprocessing is triggered only for \texttt{shl\_200}. Both methods are then iterated until the residual \eqref{eq:metric-residual} becomes smaller than the dense double-precision residual $r_{\mathrm{DP}}$, or until ten refinement steps have been completed.

Table~\ref{tab:suitesparse-status} reports the resulting iteration counts and stopping outcomes. The right-only refinement reaches the double-precision residual level within ten steps on \texttt{d\_ss}, \texttt{shl\_200}, and \texttt{bp\_400}, requiring two, two, and three steps, respectively. The left-right method reaches the same target on \texttt{d\_ss}, \texttt{shl\_200}, and \texttt{bp\_400} in three, three, and four steps, respectively. On \texttt{fs\_541\_2}, neither method reaches the double-precision baseline within ten steps; the right-only residual grows to $2.96\times 10^{20}$ and the $W$-method produces nonfinite values before the iteration limit is reached. A separate matching test between the single-precision and double-precision eigendecompositions shows that the initial single-precision right-eigenvector matrix is already too inaccurate for the refinement assumptions used in this paper: after matching eigenpairs by eigenvalue and scaling each vector optimally, the Frobenius relative error is $5.76\times 10^{-1}$ and the worst matched column has relative error of order one. Figure~\ref{fig:suitesparse} therefore omits the stopped refinement residuals for \texttt{fs\_541\_2} and shows only the initial and double-precision reference levels. On this small benchmark set, the right-only refinement remains more robust, while the $W$-method should still be viewed as a local post-processing procedure whose success depends on the quality of the left-right initialization.

\begin{table}[t]
\centering
\caption{Stopping behavior on the retained dense SuiteSparse benchmarks.}
\label{tab:suitesparse-status}
\small
\begin{tabular}{lcc}
\hline
Matrix & right-only & $W$-method \\
\hline
\texttt{d\_ss} & converged in 2 steps & converged in 3 steps \\
\texttt{fs\_541\_2} & not converged & not converged \\
\texttt{shl\_200} & converged in 2 steps & converged in 3 steps \\
\texttt{bp\_400} & converged in 3 steps & converged in 4 steps \\
\hline
\end{tabular}
\end{table}

\begin{figure}[t]
\centering
\begin{tikzpicture}
\begin{axis}[
    xbar,
    bar width=4.5pt,
    width=0.94\linewidth,
    height=0.54\linewidth,
    xmode=log,
    x dir=reverse,
    xlabel={Relative residual},
    ylabel={SuiteSparse matrix},
    ytick={1,2,3,4},
    yticklabels={d\_ss,fs\_541\_2,shl\_200,bp\_400},
    y dir=reverse,
    legend style={at={(0.5,1.02)},anchor=south,legend columns=4,font=\scriptsize},
    grid=both,
    unbounded coords=discard,
    enlarge y limits=0.18]
\addplot table[col sep=space, comment chars={@}, x index=4, y expr=\coordindex+1] {data/suitesparse_benchmarks.dat};
\addlegendentry{SP initial}
\addplot table[col sep=space, comment chars={@}, x expr={\thisrowno{0}==2 ? nan : \thisrowno{5}}, y expr=\coordindex+1] {data/suitesparse_benchmarks.dat};
\addlegendentry{right-only}
\addplot table[col sep=space, comment chars={@}, x expr={\thisrowno{0}==2 ? nan : \thisrowno{6}}, y expr=\coordindex+1] {data/suitesparse_benchmarks.dat};
\addlegendentry{$W$-method}
\addplot table[col sep=space, comment chars={@}, x index=7, y expr=\coordindex+1] {data/suitesparse_benchmarks.dat};
\addlegendentry{DP baseline}
\end{axis}
\end{tikzpicture}
\caption{Residual levels on retained dense copies of application-derived SuiteSparse matrices. For \texttt{fs\_541\_2}, the refinement bars are omitted because both refinements fail from an initialization that is already too inaccurate for iterative refinement to be meaningful.}
\label{fig:suitesparse}
\end{figure}
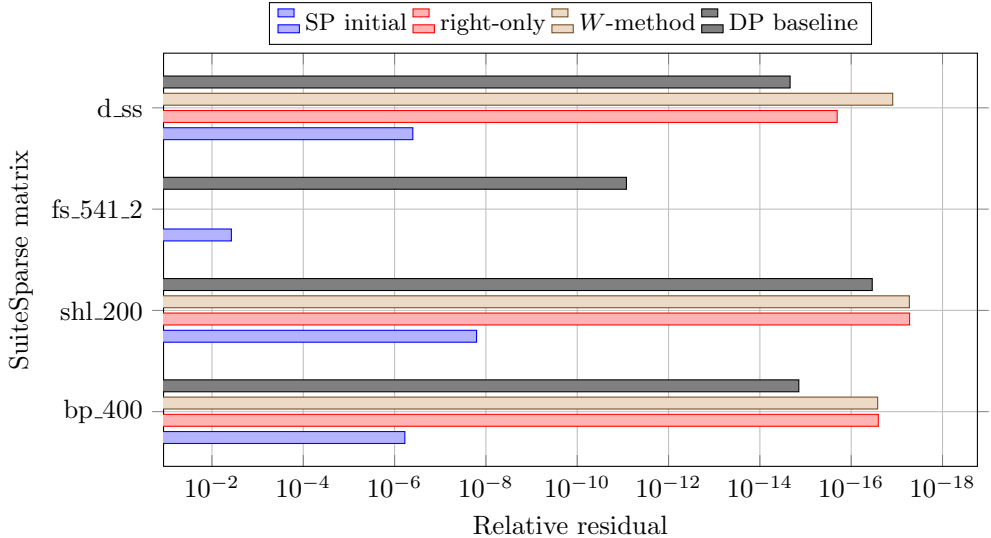

\subsection{Timing and mixed precision}

Finally, we compare the total running time of three workflows: a full double-precision eigensolve (\textbf{DP}), a full single-precision eigensolve (\textbf{SP}), and a single-precision eigensolve followed by iterative refinement in double precision (\textbf{SP+IR}). The timing matrices belong to the same real simple-spectrum family with $\alpha=0.05$, and we test the dimensions $n\in\{400,800,1600,2400,3200\}$. Measurements are obtained from the C implementation with \texttt{OMP\_NUM\_THREADS=2} and \texttt{VECLIB\_MAXIMUM\_THREADS=1}; each timing is the minimum over three runs. In the right-only case, \textbf{SP+IR} uses two right-only refinement steps with direct solves. In the left-right case, \textbf{SP+IR} uses exact biorthogonalization followed by two $W$-method steps. The total cost of preprocessing is included in \textbf{SP+IR}. These dimensions are not close to the storage limit of a 16\,GB machine for dense matrices; in this setting, repeated full eigensolves become runtime-limited before they become memory-limited. We therefore extend the curves to the largest dimension that remained practical for the full repeated workflow in the present environment. At $n=3200$, \textbf{SP+IR} is about $35\%$ faster than \textbf{DP} in the right-only regime and about $31\%$ faster in the left-right regime, while the residual after refinement remains comparable to the double-precision baseline.

\begin{figure}[t]
\centering
\begin{minipage}[t]{0.46\linewidth}
\vspace{0pt}
\centering
\begin{tikzpicture}
\begin{axis}[
    width=\linewidth,
    height=0.8\linewidth,
    xmode=log,
    ymode=log,
    xlabel={Dimension $n$},
    ylabel={Total time [s]},
    legend style={at={(0.02,0.98)},anchor=north west,font=\scriptsize},
    grid=both]
\addplot table[col sep=space, comment chars={@}, x index=0, y index=1] {data/mixed_precision_timing.dat};
\addlegendentry{DP}
\addplot table[col sep=space, comment chars={@}, x index=0, y index=2] {data/mixed_precision_timing.dat};
\addlegendentry{SP}
\addplot table[col sep=space, comment chars={@}, x index=0, y index=3] {data/mixed_precision_timing.dat};
\addlegendentry{SP+IR}
\end{axis}
\end{tikzpicture}

\small Right-only regime
\end{minipage}\hfill
\begin{minipage}[t]{0.46\linewidth}
\vspace{0pt}
\centering
\begin{tikzpicture}
\begin{axis}[
    width=\linewidth,
    height=0.8\linewidth,
    xmode=log,
    ymode=log,
    xlabel={Dimension $n$},
    ylabel={Total time [s]},
    legend style={at={(0.02,0.98)},anchor=north west,font=\scriptsize},
    grid=both]
\addplot table[col sep=space, comment chars={@}, x index=0, y index=4] {data/mixed_precision_timing.dat};
\addlegendentry{DP}
\addplot table[col sep=space, comment chars={@}, x index=0, y index=5] {data/mixed_precision_timing.dat};
\addlegendentry{SP}
\addplot table[col sep=space, comment chars={@}, x index=0, y index=6] {data/mixed_precision_timing.dat};
\addlegendentry{SP+IR}
\end{axis}
\end{tikzpicture}

\small Left-right regime
\end{minipage}
\caption{Total running time for double precision, single precision, and single precision followed by refinement; the right panel uses two $W$-method steps after exact biorthogonalization.}
\label{fig:timing}
\end{figure}

\section{Conclusion}\label{sec:conclusion}

This paper develops iterative refinement for diagonalizable non-Hermitian eigendecompositions in two regimes determined by the available input data. The main theory concerns simple eigenvalues. In the right-only regime, the update selected by the first-order derivation satisfies an exact residual identity and a quadratic residual bound. In the left-right regime, the computable driving matrix is an exact perturbation of the inverse-based one, the biorthogonality correction satisfies an exact Newton--Schulz-type error identity, and these two facts yield a local second-order estimate for the $W$-method.

The numerical experiments are organized to mirror this division. For simple eigenvalues, the observed behavior is consistent with the local analysis. On dense copies of application-derived SuiteSparse matrices, the right-only refinement reaches the dense double-precision residual level within ten steps on three of the four retained cases, whereas the left-right method requires one additional step on the successful cases and does not reach that target on \texttt{fs\_541\_2}. That example also shows that the method is not designed to recover from an initialization that already violates the basic refinement assumption. For clustered eigenvalues, the projected re-diagonalization strategy stabilizes cases in which the naive componentwise update stagnates or diverges. The mixed-precision timings indicate that a single-precision eigensolve followed by double-precision refinement can be cheaper than a full double-precision eigensolve while delivering comparable residuals.

The scope of the analysis is deliberately limited. The theory is local and restricted to simple eigenvalues, and the paper treats refinement as post-processing for an already accurate eigendecomposition. It does not analyze the size of the admissible neighborhood of initial approximations, and the clustered treatment is presented only as a stabilization extension rather than as a completed convergence theory. The experiments still rely mainly on generated matrices designed to isolate specific numerical effects, while the application-derived tests serve only as an initial robustness check. Broader validation on larger and more diverse benchmark sets remains a natural next step.

\backmatter

\bmhead{Acknowledgements}

This work was supported by JSPS KAKENHI Grant Number 25H00449.







\bibliography{references}

\end{document}